\documentclass[final,12pt]{colt2016} 

\title[Global Convergence of GROUSE]{Global Convergence of a Grassmannian Gradient Descent \\Algorithm for Subspace Estimation}
\usepackage{times}

\usepackage{algorithm}
\usepackage{algorithmic}
\usepackage{multirow,multicol}
\usepackage{enumerate}
\usepackage{etoolbox}
\newbool{DrawFigures}
\booltrue{DrawFigures} 

\usepackage{tikz, pgfplots}
\pgfplotsset{compat=newest}
\newlength{\figurewidth}
\newlength{\figureheight}

\newcommand{\R}{\mathbb{R}}

\newcommand{\E}{\mathbb{E}}
\newcommand{\bP}{\mathbb{P}}

\newcommand{\diag}{\operatorname{diag}}
\newcommand{\trace}{\operatorname{\textbf{tr}}}

\newcommand{\nn}{\nonumber}

\newcommand{\argmax}{\operatornamewithlimits{arg\ max}}

\def\ie{{\em i.e.,~}}
\def\eg{{\em e.g.,~}}

\newtheorem{condition}{Condition}

 \coltauthor{\Name{Dejiao Zhang} \Email{dejiao@umich.edu}\\
 \Name{Laura Balzano} \Email{girasole@umich.edu}\\
 \addr Electrical Engineering and Computer Science, University of Michigan
 }
 \jmlrworkshop{Accepted by AISTATS 2016}

\begin{document}

\maketitle

\begin{abstract}
\sloppypar{It has been observed in a variety of contexts that gradient descent methods have great success in solving low-rank matrix factorization problems, despite the relevant problem formulation being non-convex. We tackle a particular instance of this scenario, where we seek the $d$-dimensional subspace spanned by a streaming data matrix. 
We apply the natural first order incremental gradient descent method, constraining the gradient method to the Grassmannian. In this paper, we propose an adaptive step size scheme that is greedy for the noiseless case, that maximizes the improvement of our metric of convergence at each data index $t$, and yields an expected improvement for the noisy case. We show that, with noise-free data, this method converges from any random initialization to the global minimum of the problem. For noisy data, we provide the expected convergence rate of the proposed algorithm per iteration.}
\end{abstract}

\section{Introduction}
Low-rank matrix factorization is one of the foundational tools of signal processing, numerical methods, and data analysis. Suppose we wish to factorize a matrix $M=UW^T$, imposing orthogonality constraints on $U$ or $W$. Solving for such matrix factorizations can be computationally burdensome, and many algorithms that attempt to speed up computation are actually solving a non-convex optimization problem, therefore coming with few guarantees. 

\sloppypar{The Singular Value Decomposition (SVD) is the solution to a non-convex optimization problem, and there are several highly successful algorithms for solving it \cite{golub2012matrix}. Unfortunately, these algorithms cannot easily be extended to problems with regularizers or missing data. Recently, several results have been published with first-of-their-kind guarantees for a variety of different gradient-type algorithms on non-convex matrix factorization problems \cite{jain2013low, de2014global, armentano2014average,chen2015fast,bhojanapalli2015dropping,zheng2015convergent}. These new algorithms, being gradient-based, are well-suited to extensions of the original problem that include different cost functions or regularizers. For example, with gradient methods to solve the SVD we may be able to solve Robust PCA \cite{candes2011robust, he2012incremental, xu2010robust}, Sparse PCA \cite{d2008optimal}, or even $\ell_1$ PCA \cite{brooks2013pure} with gradient methods as well.}

Our contribution is to provide a global convergence result for $d$-dimensional subspace estimation using an incremental gradient algorithm performed on the Grassmannian, the space of all $d$-dimensional subspaces of $\R^n$. 
Subspace estimation is a special case of matrix factorization with orthogonality constraints, where we seek to estimate only the subspace spanned by the columns of the left matrix factor $U \in \R^{n\times d}$. Our result demonstrates that this gradient algorithm \emph{converges globally} almost surely, \ie it converges from any random initialization to the global minimizer. To the best of our knowledge, this is the first global convergence result for an incremental gradient descent method on the Grassmannian. When there is no noise, we propose a greedy step size scheme that maximizes the improvements on the defined metrics of convergence. Given this, we provide a rate of convergence in two parts: slower convergence in an initial phase starting from the random initialization, and then linear convergence for a local region around the global minimizer, where our results match those in \cite{balzano2014local}. For the noisy case, we propose a step-size regimen that is simply a weighted version of the step size for noise-free data, where the weights depend on the data and noise statistics. With this step size, we provide results guaranteeing monotonic improvements on the metrics of convergence in terms of expectation. 

Incremental gradient descent is our focus, motivated by streaming data applications. There are many applications of subspace estimation and tracking in medical imaging, communications, and environmental science; see more in~\cite{edelman1998geometry, balzano2014local, balzano2012handling}. 
Matrix factors with orthogonality constraints, such as those given by the SVD, are also used in several data applications: they provide a unique collection of low-dimensional projections for data visualization, capture directions of maximal variance so as to give useful insights into data structure, and allow compressed storage of massive datasets with a precise notion of loss in compression.

\section{Formulation and Related Work}

We may formulate subspace estimation as a non-convex optimization problem as follows. Let $M \in \R^{n \times N}$ be a matrix that we wish to approximate with a subspace of rank $d$, and solve:
\vspace{-5mm}

\begin{align}
\underset{U \in \R^{n\times d}, W \in \R^{N \times d}}{\text{minimize}} \quad &  \|UW^T-M\|_F^2 \label{prob:batch} \\
\text{subject to } \quad & \text{span}\left(U\right)\in  \mathcal{G}(n,d)\nn
\end{align} 
This problem is non-convex firstly because of the product of the two optimization variables $U$ and $W$ and secondly because the optimization is over the Grassmannian $\mathcal{G}(n,d)$, the non-convex set of all $d$-dimensional subspaces in $\mathbb{R}^{n}$. However, several methods\footnote{For example, the power method can solve this problem if the top $d$ singular values of $M$ are distinct~\cite{golub2012matrix}. Specifically, considering $d=1$, if the desired accuracy of the $U$ output by the power method to the global minimizer is $\epsilon^*$, and the first two singular values of $M$, $\sigma_1(M)$ and $\sigma_2(M)$ are distinct with the $\sigma_1(M) = c \sigma_2(M)$ for $c>1$, then the power method converges in $O\left(\frac{\log(1/\epsilon^*)}{\log c}\right)$ iterations.} can find the global minimizer of this problem in polynomial time under a variety of assumptions on $M$. 

In this paper, we are interested in approximating a streaming data matrix. 
At each step, we sample a column of $M$, denoted $x_t \in \R^n$. We consider the planted problem, where $x_t = v_t + \xi_t$ where $\xi_t$ is noise and $v_t$ is drawn from a continuous distribution with support on the true subspace, spanned by $\bar{U} \in \R^{n\times d}$ with orthonormal columns; $v_t = \bar{U}s_t$, $s_t \in \R^d$. When $\xi_t=0$, we wish to find the $U$ that minimizes 
\begin{equation}
\label{eq:cost} 
F(U) = \sum_{t=1}^\infty \min_{w_t} \|U w_t - x_t \|_2^2\;,
\end{equation} \ie the span of the data vectors or the range of $\bar U$, denoted $R(\bar U)$. When $\xi_t\neq 0$ we still discuss results in terms of the distance from $\bar U$.
If we consider only $t=1,\dots,N$, Problem (\ref{eq:cost}) is identical to Problem~\eqref{prob:batch}. 
The GROUSE algorithm (Grassmannian Rank-One Update Subspace Estimation) we analyze is shown as Algorithm \ref{alg:grouse}, where we generate a sequence $\{U_t\}_{t = 0, 1, \dots}$ of $n \times d$ matrices with orthonormal columns with the goal that $R(U_t)\rightarrow R(\bar{U})$ as $t \rightarrow \infty$. Each observed vector is used to update $U_t$ to $U_{t + 1}$, and we constrain the gradient descent method to the Grassmannian using a geodesic update~\cite{edelman1998geometry}. 

Because of the importance of the problem, it has been studied for decades, and there is a great deal of related work. We direct the reader to \cite{edelman1998geometry, balzano2012handling} for in-depth descriptions of algorithms and guarantees. We focus here on recent results that have \emph{global convergence guarantees to the global minimizer} and study either gradient-type algorithms, algorithms that handle streaming data, or algorithms that maintain orthogonality constraints with manifold optimization.

First we discuss incremental methods. \cite{de2014global} established the global convergence of a stochastic gradient descent method for the recovery of a positive definite matrix $M$ in the undersampled case, where the matrix $M$ is not measured directly but instead via linear measurements. They propose a step size scheme under which they prove global convergence results from a randomly generated initialization. Similarly, \cite{balsubramani2013fast} invokes a martingale-based argument to show the global convergence rate of the proposed incremental PCA method to the single top eigenvector in the fully sampled case.  In contrast, \cite{arora2013stochastic} estimates the best $d$-dimensional subspace in the fully sampled case and provides a global convergence result by relaxing the non-convex problem to a convex one. We seek to identify the $d$ dimensional subspace by solving the non-convex problem directly. Finally, our work is most related to~\cite{balzano2014local}, which provides local convergence guarantees for GROUSE in both the fully sampled and undersampled case. Our work focuses on global convergence but only in the fully sampled case; we will extend the global convergence results to the undersampled case in future work.

Turning to batch methods, \cite{RH2012, jain2013low} provided the first theoretical guarantee for an alternating minimization algorithm for low-rank matrix recovery in the undersampled case. 
Under typical assumptions required for the matrix recovery problems \cite{recht2010guaranteed}, 
they established geometric convergence to the global optimal solution. Earlier work \cite{keshavan2010matrix, ngo2012scaled} considered the same undersampled problem formulation and established convergence guarantees for a steepest descent method (and a preconditioned version) on the full gradient, performed on the Grassmannian. {\cite{chen2015fast,bhojanapalli2015dropping,zheng2015convergent} considered low rank semidefinite matrix estimation problems, where they reparamterized the underlying matrix as $M = UU^T$, and update $U$ via a first order gradient descent method. However, all these results require batch processing and a decent initialization that is close enough to the optimal point, resulting in a heavy computational burden and precluding problems with streaming data.}
We study random initialization, and our algorithm has fast, computationally efficient updates that can be performed in an online context. 

Lastly, several convergence results for optimization on general Riemannian manifolds, including several special cases for the Grassmannian, can be found in \cite{absil2009optimization}. Most of the results are very general; they include global convergence rates to local optima for steepest descent, conjugate gradient, and trust region methods, to name a few. We instead focus on solving the problem in \eqref{eq:cost} and provide global convergence rates to the global minimum.

\section{Convergence analysis}
\label{sec:fullconvg}

We analyze Algorithm~\ref{alg:grouse}. At each step, the algorithm receives a vector $x_t  = v_t + \xi_t \in \R^n$ such that $v_t = \bar{U}s_t$, $s_t \in \R^d$ and $\xi_t$ is zero mean Gaussian noise. The algorithm then outputs an $n \times d$ matrix $U_t$ with orthonormal columns at each iteration. We wish to recover $\bar U$, \ie the minimizer of Equation~\eqref{eq:cost} when there is no noise. 
%
%
%
We would like to emphasize that in this scenario in a real application one would use the ISVD or a Gram-Schmidt procedure, but we seek convergence results for the Grassmannian gradient descent algorithm so that extensions can be made; \eg we may regularize the cost function or we may minimize some other function of the data.  Reliable global convergence of the GROUSE algorithm 
has been observed empirically, despite the fact that the algorithm is solving a non-convex problem and operating on a non-convex manifold.

Algorithm~\ref{alg:grouse} takes each vector $x_t$, forms the gradient of $\min_w \|U w - x_t \|_2^2$, and takes a step in the direction of the negative gradient. The step is taken along the Grassmannian, the manifold of all $d$-dimensional subspaces of $\R^n$, and according to the step size described and justified below. In words, the algorithm works as follows: First we project our data vector onto the current subspace iterate to get the projection $p_t$. Then we calculate the residual $r_t$. The update to our subspace estimate $U_t$ then requires only the addition of a rank-one matrix, as can be seen in Equation~(\ref{eq:gpupdate}). This update is derived and explained in further detail in~\cite{balzano2010online, edelman1998geometry}. The rank-one update tilts $U_t$ to no longer contain $p_t$ but instead contain a linear combination of $p_t$ and $r_t$; in other words, it moves $U_t$ towards the observation $v_t$. 

\begin{algorithm}
\caption{GROUSE: Grassmannian Rank-One Update Subspace Estimation} \label{alg:grouse}
\begin{algorithmic}
\STATE{Given $U_0$, an $n \times d$ matrix with orthonormal columns, with $0<d<n$;}
\STATE{Set $t:=0$;}
\REPEAT 
\STATE{Given observation $x_t = v_t + \xi_t$ for $v_t \in R(\bar U)$;}
\STATE{Define $w_t := \arg \min_w \|U_t w - x_t\|_2^2$;}
\STATE{Define $p_t := U_t w_t$; 
$r_t := x_t - U_t w_t$; }
\STATE{Using step size

\vspace{-3mm}
\begin{equation}
\theta_t =  \arctan\left((1-\alpha_t)\frac{\| r_t\|}{\|p_t\|}\right)\;,
\label{eq:theta}
\end{equation} 
where $\alpha_t = c \frac{\sigma^2}{1 + \sigma^2}\left(1 - \frac{d}{n}\right)\frac{\|x_t\|^2}{\|r_t\|^2}$ where $c > 0$ and $\sigma^2$ denotes the upper bound for the noise level (Condition \ref{cond:fullnoise}), update with a gradient step on the Grassmannian:}
\begin{equation} \label{eq:gpupdate}
U_{t+1} := U_t + \left(\frac{y_t}{\|y_t\|} - \frac{p_t}{\|p_t\|}\right)\frac{w_t^T}{\|w_t\|}
\end{equation} 

\vspace{-2mm}
\noindent where 
\vspace{-2mm}
$$\frac{y_t}{\|y_t\|} = \left[ \cos (\theta_t) \frac{p_t}{\|p_t\|} + \sin(\theta_t) \frac{r_t}{\|r_t\|} \right] $$

\vspace{-2mm}
\STATE{$t:=t+1$;}
\UNTIL{termination}
\end{algorithmic}
\end{algorithm}

Before we present our main results on the convergence of the GROUSE algorithm, we first call out the following definitions and condition that will be used throughout our analysis.  
\begin{definition}[Principal Angles]
We use $\phi_i\left(\bar{U}, U_t\right), i = 1, \dots, d$ to denote the principal angles between subspaces $\text{R}(U_t)$ and $\text{R}(\bar U)$, which are defined [\cite{stewart1990matrix}, Chapter 5] by  $\cos \phi_i(\bar{U},U_t) = \sigma_{i}(\bar U^T U_t)$. 
\label{defn:angles}
\end{definition}

\begin{definition}[Determinant similarity]
Our first metric is $\zeta_t \in [0,1]$, which measures the similarity between two subspaces and is defined as
\begin{equation}
\zeta_t := \det(\bar U^T U_t U_t^T \bar U)  = \prod_{i=1}^d \cos^2 \phi_i(\bar{U}, U_t)\;.
\label{zetadefn}
\end{equation}
\label{defn:detdiscrepancy}
\end{definition}

\vspace{-5mm}

\begin{definition}[Frobenius norm discrepancy]
Our second metric is $\epsilon_t \in [0,d]$, which measures the discrepancy between $R(U_t)$ and $R(\bar U)$, and is defined as
	\begin{equation}
		\epsilon_t := \sum_{i = 1}^{d}\sin^2{\phi_i(\bar U, U_t)} =  d - \|\bar U^T U_t\|_F^2  \;.
		\label{epsilondefn}
	\end{equation}
    \label{defn:epsdiscrepancy}
\end{definition}

\vspace{-5mm}
\begin{condition}
	The inputs of GROUSE are $x_t = v_t + \xi_t$ where $v_t = \bar U s_t$ with $\mathbb{E}s_t = 0, \text{Cov}(s_t) = \mathbb{I}_{d}$, and $\xi_t$ is a Gaussian random vector with entries being independently normal random variables such that $\mathbb{E}\left[\|\xi_t\|^2 / \|v_t\|^2 \big\lvert v_t\right] \leq \sigma^2$. Further, we assume the energy of the underlying signals are finite, \ie $\|v_t\|^2 < \infty$.    
	\label{cond:fullnoise}
\end{condition}

\subsection{Optimal Adaptive Step Size} 
\label{sec:step_size}
In this section, we first derive a greedy step size scheme for each iteration $t$ that maximizes the improvement on the defined metrics ($\epsilon_t,\zeta_t$) of convergence for the noiseless case, \ie $x_t = v_t$. Let $v_{t,\parallel}$ and $v_{t, \perp}$ denote the projection and residual of $v_t$ onto $R(U_t)$. Then after each update we have the following (Appendix \ref{sec:proof_of_supporting_theory}):
\begin{subequations}
\begin{align}
	&\frac{\zeta_{t + 1}}{\zeta_t} = \left(\cos\theta_t + \frac{\|v_{t, \perp}\|}{\|v_{t,\parallel}\|}\sin\theta_t\right)^2 \label{eq0:eps_diff}\\
	&\epsilon_t - \epsilon_{t + 1} = \frac{\left\|\bar U^T y_t\right\|^2}{\|y_t\|^2} - \frac{\|\bar U^T v_{t, \parallel}\|^2}{\|v_{t, \parallel}\|^2} \label{eq0:zeta_ratio}
\end{align}	
\end{subequations}
with $\frac{y_t}{\|y_t\|} = \frac{v_{t,\parallel}}{\|v_{t,\parallel}\|}\cos\theta_t + \frac{v_{t, \perp}}{\|v_{t, \perp}\|}\sin\theta_t$. It follows that 
\begin{equation}
	\theta^{\ast}_t = \argmax_{\theta} \frac{\zeta_{t + 1}}{\zeta_t} = \arctan \left(\frac{\|v_{t,\perp}\|}{\|v_{t,\parallel}\|}\right) \nonumber
	\label{eq:stepsize_noiseless}
\end{equation}
This is equivalent to (\ref{eq:theta}) for the noise-free case setting $\alpha_t = 0$. Using $\theta_t^{\ast}$, we obtain monotonic improvement on the determinant increment $\frac{\zeta_{t + 1}}{\zeta_t} = 1 + \frac{\|v_{t, \perp}\|^2}{\|v_t\|^2} \geq 1$.  For the Frobenius norm discrepancy, we obtain $\epsilon_{t + 1} - \epsilon_t = 1 - \frac{\left\|\bar U^T v_{t,\parallel}\right\|^2}{\|v_{t,\parallel}\|^2}$; that is, $\epsilon_t$ also achieves its maximal improvement. Therefore, when there is no noise in the observations, the proposed step size scheme described by (\ref{eq:theta}) implies greedy learning rates with respect to the defined metrics ($\epsilon_t, \zeta_t$) of convergence.

For the noisy case, we propose a weighted step size schedule by restricting $\alpha_t \in (0,1]$ with the goal that $\alpha_t \rightarrow 1$ as our estimated subspace $R(U_t)$ gradually converges to the true subspace $R(\bar U)$. The intuition behind this strategy is that, choosing the step size in Equation (\ref{eq:theta}), the update of GROUSE follows as
\begin{equation}
	U_{t + 1} =  U_t + \left(\frac{p_t + (1 - \alpha_t) r_t}{\|p_t + (1 - \alpha_t) r_t\|} - \frac{p_t}{\|p_t\|}\right)\frac{w_t^T}{\|w_t\|}  \nonumber
	\label{eq:adapt_update}
\end{equation}
for which we have, if the noise is Gaussian distributed, $\|r_t\|^2 \sim \|v_{t,\perp}\|^2 + (1 - d/n)\|\xi_t\|^2$ (where by $a \sim b$ we mean $a$ concentrates around $b$), hence the noise part will gradually dominate the projection residual as $R(U_t)\rightarrow R(\bar U)$. It is therefore natural for us to consider incorporating less and less of the residual information into $R(U_t)$ over time. Therefore, we propose the following schedule for $\alpha$: 
\begin{equation}
	\alpha_t = 1 - \frac{\|v_{t,\perp}\|^2}{\|r_t\|^2} = \frac{c\sigma^2}{1 + \sigma^2}\left(1 - \frac{d}{n}\right)\frac{\|x_t\|^2}{\|r_t\|^2}  
	\label{eq:alpha_size}
\end{equation}
where $c > 0$. As we will show in Section \ref{sec:supportingresults}, with this weighted learning rate scheme, we obtain improvements in expectation on both $\zeta_t$ and $\epsilon_t$. 


\subsection{Convergence Without Noise} 
\label{sec:conv_noiseless}
In this section, we consider the noise-free case, that is $x_t = v_t$ and $v_t \in R(\bar U)$. The step size (Eq (\ref{eq:theta})) used in this section has $\alpha_t = 0$ for all iterations. We provide analysis of the algorithm in two separate phases. In the first phase the GROUSE algorithm will converge to a local region of the global optimal point from a random initialization within $O(d^3 log(n))$ iterations. From there, in the second phase GROUSE converges linearly to the optimal point. In each phase we use a different metric of convergence, which helps us obtain an overall faster convergence rate as compared to other work. The convergence rate with respect to only either determinant \cite{de2014global} or Frobenius norm discrepancy \cite{jain2013low} is either much slower within the local region \cite{de2014global} or slower in an initial phase from random initialization \cite{jain2013low}. This is demonstrated numerically in Figure \ref{fig:deteps_rate}.

\begin{theorem}[Global Convergence of GROUSE]
Suppose Condition \ref{cond:fullnoise} and that no noise is contained in the observations, \ie $x_t = v_t$. Let $\epsilon^*>0$ be the desired accuracy of our estimated subspace using the metric 
 in Definition~\ref{defn:epsdiscrepancy}. Initialize the starting point $U_0$ of GROUSE as the orthonormalization of an $n\times d$ matrix with entries being standard normal variables. Then for any $\rho, \rho'>0$, after 
\begin{align}
K &\geq K_1 + K_2 \nn \\
&= \left(\frac{d^3}{\rho'} + d\right)\mu_0 \log (n) + 2 d \log \left( \frac{1}{\epsilon^* \rho}\right) 
\label{eq:Kmain}
\end{align} iterations of GROUSE (Algorithm~\ref{alg:grouse}), 
    \begin{align}
 \bP\left(\epsilon_K \leq \epsilon^*\right) \geq 1 - \rho'-\rho\;.
    \end{align}
    where $\mu_0 = 1 + \frac{\log \frac{(1 - \rho')}{C} + d\log (e/d)}{d\log n}$ with $C > 0$ a constant approximately equal to $1$. 
	\label{thm:global}
\end{theorem}

The proof of Theorem \ref{thm:global} is a direct combination of our analysis in two phases of the algorithm, stated in Theorem \ref{thm:detconvgrate} and Theorem \ref{thm:localepsconvg} below. 

\begin{theorem}[Initial convergence of the determinant similarity $\zeta_t$ to $\frac{1}{2}$] Under the conditions of Theorem \ref{thm:global}, for any $\rho'\in (0,1)$, after 
    \begin{equation}
    	K_1 \geq \left( \frac{d^3}{\rho'} + d\right)\mu_0 \log(n) \nn 
    \end{equation}
    iterations of GROUSE (Algorithm \ref{alg:grouse}), 
    \begin{equation}
    	\mathbb{P}\left(\zeta_{K_1} \geq \frac{1}{2}\right) \geq 1 - \rho' \; \nn
    \end{equation}
    where $\mu_0$ is the same as that in Theorem \ref{thm:global}.
\label{thm:detconvgrate}
\end{theorem}
Analyzing the determinant similarity turns out to be the key to proving convergence in this initial phase of GROUSE. The determinant similarity increases quickly toward 1 in the first phase. This also gives insight into how the GROUSE algorithm manages to seek the global minimum of a non-convex problem formulation: GROUSE is not attracted to stationary points that are not the global minimum. For our problem, all other stationary points $U_{stat}$ have $\det(\bar U^T U_{stat} U_{stat}^T \bar U) = 0$, because they have at least one direction orthogonal to $\bar U$~\cite{balzano2012handling}. If the initial point $U_0$ has determinant similarity with $\bar{U}$ strictly greater than zero, and GROUSE increases the determinant similarity monotonically (as we mentioned in Section \ref{sec:step_size} and prove in Section \ref{sec:supportingresults}), then we are guaranteed to stay away from other stationary points. Since we initialize GROUSE using $U_0$ uniformly from the Grassmannian, as the orthonormal basis of a random matrix $V\in R^{n\times d}$ with entries being independent standard Gaussian random variables, we guarantee $\zeta_0 > 0$ with probability one. 

\begin{theorem}[Local convergence of the Frobenius norm discrepancy $\epsilon_t$ to 0]
Suppose at iteration $k$ we have $\zeta_k \geq 1/2$. Then for any $\rho \in (0,1)$ and given accuracy $\epsilon^{\ast}$, after $$K_2 \geq 2 d \log \left( \frac{1}{\epsilon^* \rho}\right)$$ additional iterations of GROUSE Algorithm~\ref{alg:grouse}, we have 
\begin{equation}
\bP( \epsilon_{k+K_2} \leq \epsilon^*) \geq 1-\rho\;. \nn 
\end{equation}
\label{thm:localepsconvg}
\end{theorem}
In the first phase, we require $O\left(d^3\log(n)/\rho'\right)$ iterations to reach the local region of the global minimum, where $1 - \rho'$ is the probability with which we'll reach the local region. In simulations (Section \ref{sec:numerical}, Figure~\ref{fig:Knfree}) with isotropic Gaussian data vectors from the subspace, we actually see that $O(d^3\log(n))$ iterations are many more than enough to reach the local region, without fail. Our analysis, though, only requires zero-mean uncorrelated identically distributed random data vectors. Bounds on higher moments may admit a tighter analysis, which we leave for future work. 

The second phase only requires $O(d \log (1/\epsilon^* \rho))$ iterations to converge to $\epsilon^*$ accuracy in the Frobenius norm discrepancy metric given in Definition~\ref{defn:epsdiscrepancy}. This result is true to what we see in practice, as you can see in Figure~\ref{fig:Knfree}. The analysis behind this result provides a tighter version of [\cite{balzano2014local}, Theorem 3.2] that both grows the local region of convergence and (slightly) improves the rate to be less dependent on the current value of $\epsilon_t$.

\subsection{Convergence With Noise}
\label{sec:conv_noise}
In this section, we study the convergence behavior of GROUSE with noise in each observation. Unlike the noise-free case, here we only provide \emph{expected monotonic} improvements of our convergence metrics. As we prove in the appendix, the results we present here also imply 
the corresponding ones for the noiseless data. 
\begin{theorem}[Expected convergence rate of the determinant similarity $\zeta_t$]
	Given Condition \ref{cond:fullnoise} is satisfied, after one iteration of GROUSE we have the following improvement of the determinant similarity in expectation:
	\begin{equation}
		\mathbb{E}\left[\zeta_{t + 1}\bigg\lvert U_t\right] \geq \left(1 + \beta_0\frac{1 - \zeta_t}{d}\left(1 - \frac{\sigma^2}{\frac{1 - \zeta_t}{d} + \sigma^2}\right)\right)\zeta_t\nn
	\end{equation}
	where $\beta_0 = \frac{1}{1 + \frac{d}{n}\sigma^2}$.
	\label{thm:detratio_exp}
\end{theorem}
This theorem implies that the expected convergence rate of determinant similarity is damped by the presence of noise. To be more specific, rewrite the expected improvement as $\mathbb{E}\left[\zeta_{t + 1}\big\lvert U_t\right] \geq \left(1 + \frac{\beta_0}{(1 - \zeta_t)/d + \sigma^2}\left(\frac{1 - \zeta_t}{d}\right)^2\right)\zeta_t$. We can see that, comparing with the noiseless case, for small SNR (large $\sigma^2$), the expected increment on $\zeta_t$ is approximately scaled by $\frac{1 - \zeta_t}{d} < \frac{1}{d}$. Hence the theoretical bound on the iterations necessary to achieve given accuracy $\zeta^{\ast}$ in the small SNR case should roughly be at least $d$ times that required by the noiseless case.  For large SNR (small $\sigma^2$), the expected convergence rate is close to that of the noise-free case, as long as $\zeta_t$ is not too close to $1$. Therefore, the required iterations to arrive at the local region of the true subspace should be close to that in the noiseless case. We show the corresponding numerical illustrations in Figure \ref{fig:deteps_rate} and Figure \ref{fig:K12_noise}.  

\begin{theorem}[Expected convergence rate of the Frobenius norm discrepancy $\epsilon_t$]
	Under Condition \ref{cond:fullnoise}, we obtain the following upper bound on the decrease of Frobenius norm discrepancy $\epsilon_t$ in expectation:
	\begin{equation}
		\mathbb{E}\left[\epsilon_{t + 1}\big\lvert U_t\right] \leq \left(1 -  \frac{\beta_0}{d}\left(\cos^2\phi_{t,d} - \frac{\beta_1\sigma^2}{\frac{\epsilon_t}{d} + \beta_1\sigma^2}\right)\right)\epsilon_t   \nn
	\end{equation}
	where $\beta_0 = \frac{1}{1 + \frac{d}{n}\sigma^2}$, $\beta_1 = 1 - \frac{d}{n}$, and $\phi_{t,d}$ is the largest principal angle between $R(U_t)$ and $R(\bar U)$.
	\label{thm:eps_diff_exp}
\end{theorem}
As indicated by Theorem \ref{thm:detratio_exp}, the expected convergence rate will slow down as $\zeta_t$ increases. However, the above theorem implies that for large SNR (small $\sigma^2$), once we enter the local region of the true subspace, the convergence rate of the Frobenius discrepancy will take over. Specifically, when $\cos^2\phi_{t,d} > 1/2$, the convergence rate of $\epsilon_t$ can be bounded from below by $1 - \left(\frac{1}{2} - \frac{1}{d}\right)\frac{1}{d}$ as long as $\epsilon_t \geq d^2\sigma^2$. Therefore, an implication of Theorem \ref{thm:eps_diff_exp} is that GROUSE will converge to a ball centered on the true subspace whose radius is determined by the noise level and subspace dimension. The convergence rate will slow as GROUSE approaches this ball. On the other hand, since $1 - \epsilon_t \leq \cos^2\phi_{t,d} \leq 1 - \epsilon_t/d$, by a simple calculation we can see that for small SNR (large $\sigma^2$), the fast local convergence never kicks in. In that case, we only study the convergence behavior of GROUSE in terms of the determinant similarity $\zeta_t$. 

As we mentioned previously, with noise the improvement is not monotonic for either determinant similarity ($\zeta_t$) or Frobenius norm discrepancy ($\epsilon_t$). This is a hurdle to pass before we can provide similar global convergence results as we obtained for the noise-free case (Theorem \ref{thm:global}). However, by leveraging techniques in stochastic process theory, it might be possible to establish asymptotic convergence results or even non-asymptotic convergence results in terms of the number of iterations required before GROUSE first achieves a given accuracy. We leave this as future work.

\section{Supporting Theory}
\label{sec:supportingresults}
We first call out the following lemma to quantify the expectation of the determinant similarity between our random initialization and the true subspace. For convenience, we will drop the subscript of all terms except $\epsilon_t$, $\zeta_t$ and $U_t$ hereafter.

\begin{lemma}\cite{nguyen2014random}
Initialize the starting point $U_0$ of GROUSE as the orthonormalization of an $n\times d$ matrix
with entries being standard normal variables. Then 
\begin{equation}
  	\mathbb{E}[\zeta_0] = \mathbb{E}\left[\det(U_0^T\bar U\bar U^TU_0)\right] = C \left(\frac{d}{n e}\right)^d \nonumber
  	\label{eq:exp_zeta0}
 \end{equation} 
 where $C>0$ is a constant approximately equal to $1$. 	 
	\label{lem:exp_zeta0}
\end{lemma}
As we mentioned in Section \ref{sec:step_size}, both the determinant similarity $\zeta_t$ and the Frobenius discrepancy $\epsilon_t$ improve monotonically in the noiseless case. We formally present this in the following lemma. 
\begin{lemma}[Monotonic results for the noiseless case] When there is no noise, given the step size in Eq (\ref{eq:theta}), after one update of GROUSE we obtain 
	\begin{align}
		\frac{\zeta_{t + 1}}{\zeta_t} = 1 + \frac{\|v_{\perp}\|^2}{\|v_{\parallel}\|^2} \;, \quad \text{and}\quad \epsilon_t - \epsilon_{t + 1} = 1 - \frac{\|\bar U\bar U^T v_{\parallel}\|^2}{\|v_{\parallel}\|^2}  \nn
	\end{align}
	where $v_{\parallel}$ and $v_{\perp}$ denote the projection and residual of $v$ onto $R(U_t)$. 
	\label{lem:mono}
\end{lemma}

For the noisy case, we provide the following lemmas, which are the intermediate results that allow us to establish the expected improvements on both $\zeta_t$ and $\epsilon_t$ in Section \ref{sec:conv_noise}. 

\begin{lemma}
Given Condition \ref{cond:fullnoise} is satisfied, after one update of GROUSE we obtain the following  
\begin{equation}
	\mathbb{E}\left[\zeta_{t + 1}\big\lvert U_t\right] \geq \left(1 +  \mathbb{E}\left[(1 - \alpha)^2\frac{\|r\|^2}{\|p\|^2}\bigg\lvert U_t\right]\right)\zeta_t\nn
\end{equation}
\label{lem:det_ratio_exp}	
\end{lemma}

\begin{lemma}
After one iteration of the GROUSE algorithm, we have the following
	\begin{align}
		\mathbb{E}\left[\epsilon_t - \epsilon_{t + 1}\big\lvert U_t\right] = \mathbb{E}\left[1 - \mathcal{R} - \frac{\|\bar U\bar U^T p\|^2}{\|p\|^2} \bigg\lvert U_t\right]\nn 
	\end{align}
	where $\mathcal{R} = \frac{\|(I - \bar U\bar U^T)(\xi - \alpha r)\|^2}{\|v + \xi - \alpha r\|^2}$.
	\label{lem:epsdiff_expr}
\end{lemma}
According to our definition of $\alpha$ in Section \ref{sec:step_size}, we can see that when $R(U_t)$ is not close to $R(\bar U)$, $1 - \alpha$ is large, as is $\|r\|^2 / \|p\|^2$. Therefore, Lemma \ref{lem:det_ratio_exp} implies that the expected convergence rate of the determinant similarity ($\zeta_t$) is faster in the first phase. For the Frobenius norm discrepancy ($\epsilon_t$), comparing to the noiseless case where $p = v_{\parallel}$, Lemma \ref{lem:epsdiff_expr} implies that we obtain monotonic expected decrease in Frobenius norm discrepancy as long as we are outside a ball centered on the true subspace. This ball shrinks as $\sigma^2 \to 0$, with no such constraint for $\sigma^2=0$. As we approach this ball, the expected convergence rate slows. 

\section{Numerical Results}
\label{sec:numerical}
With our plots we illustrate why the two analysis approaches allow us to prove rates in both phases of GROUSE. For each numerical result in this section, we initialize GROUSE with orthonormalized Gaussian matrices with entries iid $\mathcal{N}(0,1)$.  The underlying subspace of each trial is set to be a sparse subspace, as the range of an $n\times d$ matrix $\bar U$ with sparsity on the order of $\frac{\log(n)}{n}$. We generate the coefficient matrix $W$ with entries i.i.d $\mathcal{N}(0,1)$.  For the noisy case, we then normalize the columns of the underlying matrix $\bar U W$ and add a noise matrix $N$, with $N_{ij} \sim \mathcal{N}(0,\sigma^2/n)$. In the noisy case, we run GROUSE with the step size described in Equation (\ref{eq:alpha_size}), where we set $c$ to its expected value $1$. 

\setlength\figurewidth{0.90\columnwidth} 
\setlength\figureheight{0.30 \columnwidth} 
\begin{figure}[tb]\scriptsize
	\begin{center}
		\ifbool{DrawFigures}{
			\input{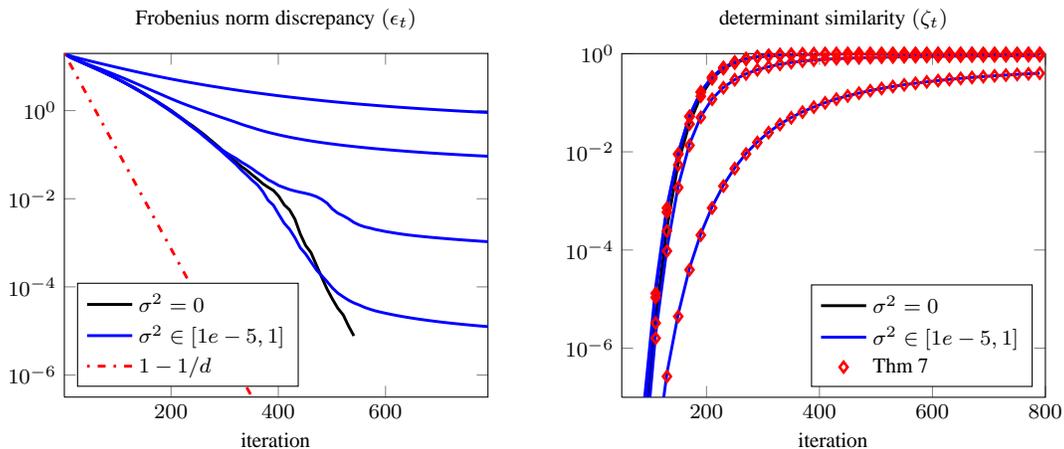}
		}{
			\textcolor{green}{\rule{\figurewidth}{\figureheight}}
		}
	\end{center}
	\caption{Illustration of expected convergence bounds given by Theorem \ref{thm:eps_diff_exp} (left) and Theorem \ref{thm:detratio_exp} (right) over $100$ trials. In this simulation, $n = 2000$, $d = 20$ and $\sigma^2 \in \left\{0,1e-5,1e-3,1e-1,1\right\}$. The red dashed line indicates the linear convergence rate, while the diamonds denote the lower bound on expected convergence rate in Theorem \ref{thm:detratio_exp}. We can see that the convergence rate of each phase slows down in the noisy case. However, when $\sigma^2$ is small, the convergence behavior of GROUSE similar to that of the noise-free case. We get a faster convergence rate of $\zeta_t$ in the initial phase an almost linear convergence of $\epsilon_t$ in the local region of $R(\bar U)$.}
	\label{fig:deteps_rate}
\end{figure}

As is demonstrated in Figure \ref{fig:deteps_rate}, when there is no noise in the observations or the SNR is large enough, the determinant similarity ($\zeta_t$) increases quickly in the initial phase, while the Frobenius norm discrepancy ($\epsilon_t$) decreases slowly. Then in a local region of the true subspace, our accurate bound on the fast convergence of the Frobenius norm discrepancy takes over. However, if the SNR is small, the convergence rate of the Frobenius norm discrepancy slows down; in this scenario we only study the convergence of GROUSE in terms of the determinant similarity. In Figure \ref{fig:deteps_rate}, we show that the convergence rate of determinant similarity will also slow down as we increase the magnitude of $\sigma^2$, however, the convergence rate described in Theorem \ref{thm:detratio_exp} is still tight. This allows us to obtain a good enough approximation of the number of iterations required to reach a ball around $R(\bar U)$, which is captured by $K_1$ alone in this case.
\setlength\figurewidth{0.80\columnwidth} 
\setlength\figureheight{0.65 \columnwidth} 
\begin{figure}[H]\scriptsize
	\begin{center}
		\ifbool{DrawFigures}{
			\input{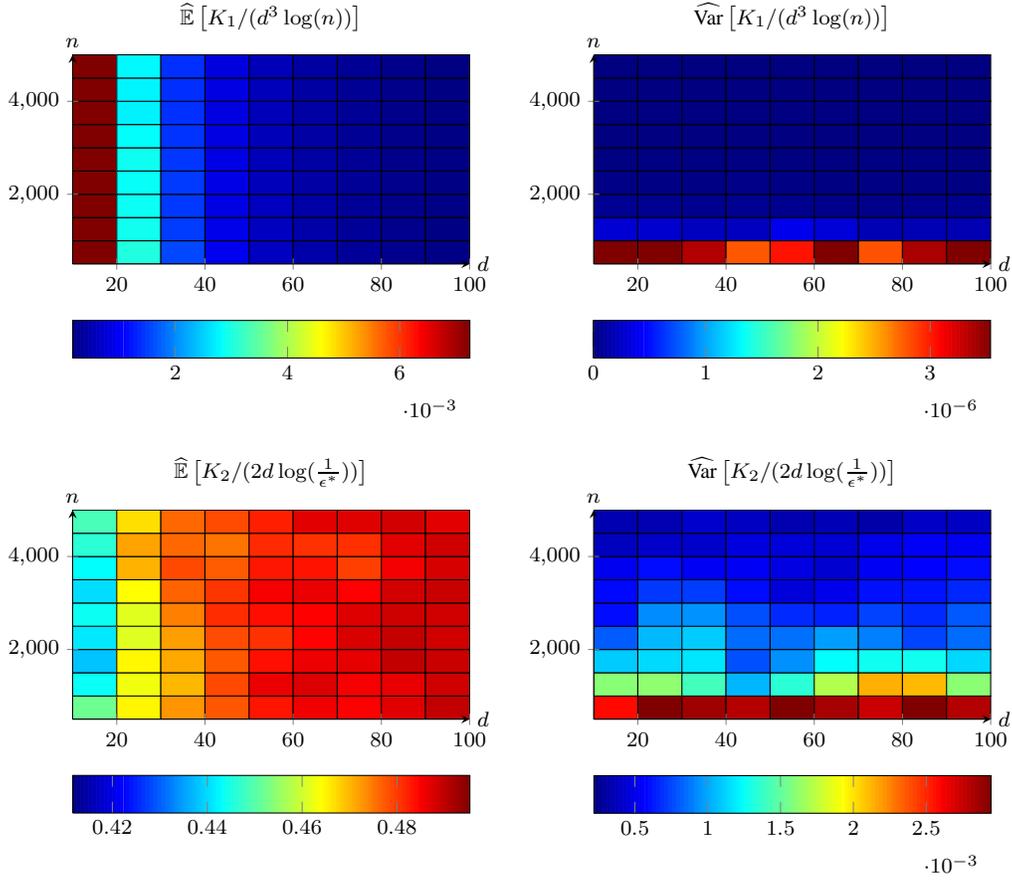}
		}{
			\textcolor{green}{\rule{\figurewidth}{\figureheight}}
		}
	\end{center}
	\caption{Illustration of the bounds on $K_1$ and $K_2$ compared to their values in practice, averaged over $50$ trials with different $n$ and $d$. We show the ratio of $K_1$ to the bound $d^3\log(n)$ in the initial phase (left)
and the ratio of $K_2$ to the bound $d \log(1/\epsilon^*)$ in the local phase (right).}
\label{fig:Knfree}
\end{figure}
We next examine the tightness of our theoretical values of $K_1$ and $K_2$ for noiseless convergence in Figure~\ref{fig:Knfree}. We run GROUSE to convergence for a required accuracy $\epsilon^* = 1e-4$ and divide the iterations into $K_1$, the number to reach $\zeta_t > \frac{1}{2}$, and $K_2$, the remaining number to reach $\epsilon_t < \epsilon^*$. We show the ratio of $K_1$ to the bound $d^3\log(n)$ in the initial phase (top plot) and the ratio of $K_2$ to the bound $d \log(1/\epsilon^*)$ in the local phase (bottom plot). We run $50$ trials and show the mean and variance. We can see that the value for $K_1$ is very loose. On the other hand, the value for $K_2$ is very accurate; $O(d \log(1/\epsilon^*))$ iterations are required to get to accuracy $\epsilon^*$. 

Finally, we examine the tightness of approximated $K_1$ and $K_2$ for the noisy case in Fig \ref{fig:K12_noise}. As we mentioned in Section \ref{sec:conv_noise}, for small SNR (large $\sigma^2$), the necessary number of iterations to achieve the given accuracy should be roughly $d$ times that required by the noise-free case, while for large SNR (small $\sigma^2$), this ratio would be less.  For large SNR, we first run GROUSE to reach the local region of the true subspace, \ie $\zeta_{K_1} \geq \frac{1}{2}$, and record $K_1$; from this point we run GROUSE to converge to $\epsilon^{\ast} = \tau_1\frac{d^2}{n}\sigma^2$ and then record $K_2$ and compare it with that required by the noise-free case.  For small SNR (large $\sigma^2$), we only numerically examine the convergence rate of the first phase, \ie necessary iterations to achieve the given accuracy $\zeta_{K_1} \geq \left(1 - \tau_2\frac{d}{n}\sigma^2\right)^{d} \approx e^{-\tau_2 d^2\sigma^2/n}$. As we can see in Figure \ref{fig:K12_noise}, we test $K_1$ versus $O(d^3\log(n))$, and as in noiseless case the bound on $K_1$ is loose. For small noise, the bound on $K_2$ is tight and stable. 

\setlength\figurewidth{0.80\columnwidth} 
\setlength\figureheight{0.60 \columnwidth} 
\begin{figure}[H]\scriptsize
	\begin{center}
		\ifbool{DrawFigures}{
			\input{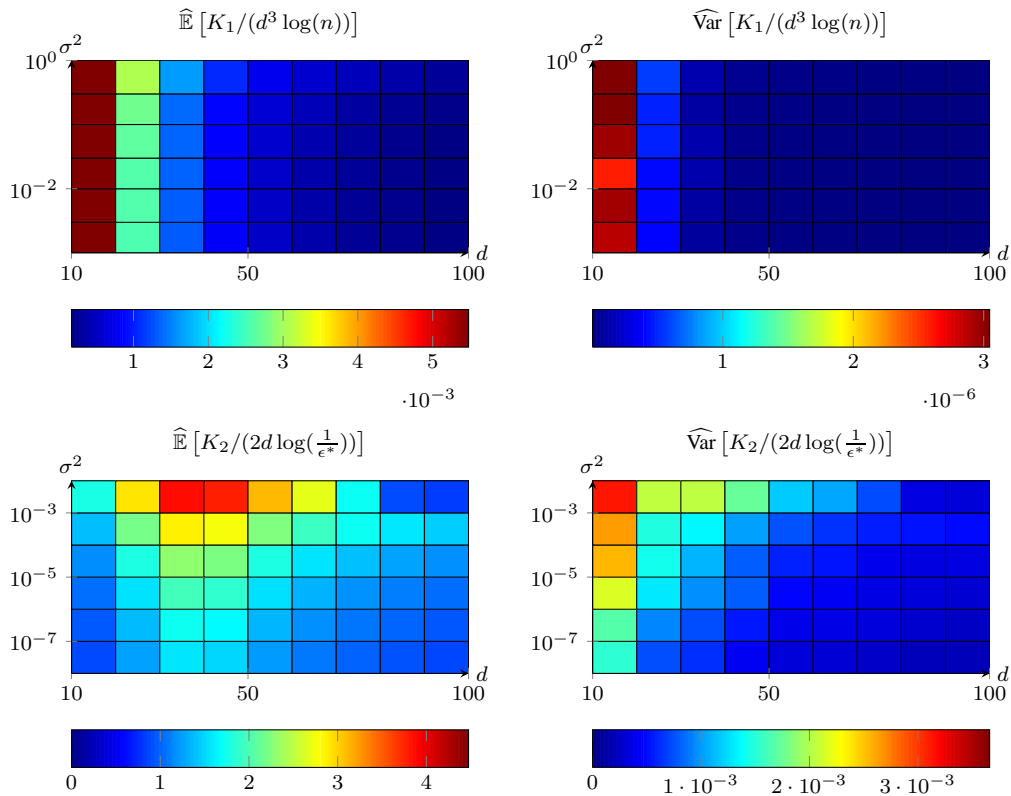}
		}{
			\textcolor{green}{\rule{\figurewidth}{\figureheight}}
		}
	\end{center}
	\caption{Illustration of the bounds on $K_1,K_2$ over $d$ and $\sigma$ with fixed $n$, averaged over $50$ trials. In this simulation, $n = 5000$ with $d$ ranging from $10$ to $100$ and $\sigma^2$ chosen from $1e$-3 to $1$ (above) and $1e$-8 to $1e-2$ (below).  We run GROUSE to converge to $\min\left\{1/2, e^{- \tau_2 d^{2}/n}\right\}$ and record the corresponding $K_1$. For large SNR (small $\sigma^2$) we first run GROUSE to converge to $1/2$ and record $K_1$, then let GROUSE further converge to $\epsilon^{\ast} = \max\left\{\sigma^2, \tau_1\frac{d^2}{n}\sigma^2\right\}$ and record $K$ that gives $K_2 = K - K_1$. For this plot we set both $\tau_1, \tau_2$ be $\log(d)$.}
	\label{fig:K12_noise}
\end{figure}

\section{Conclusion}
This paper has provided the first global convergence result for an incremental gradient descent method on the Grassmannian for noise-free data. For optimizing a particular cost function \eqref{eq:cost} in the noiseless case, we showed that the gradient algorithm converges from any random initialization to the global minimizer. Our novel analysis shows the convergence happens in two phases: the initial convergence and the local convergence. In the initial phase, we provided a very loose bound on the number of iterations $K_1 = O(d^3\log(n)/\rho')$ required to get to a local region of the global minimizer from the random initialization with probability $1-\rho'$. In fact, this phase usually takes many fewer iterations and reaches the local region in all empirical trials. In the local phase for the noiseless case, we provided a very tight bound for the required iterations $K_2 = O(d \log(1/\epsilon^*)$ to achieve a final desired accuracy of $\epsilon^*$. 

When the observations contain noise, we establish a rate of expected improvement of both of our metrics $\zeta_t$ and $\epsilon_t$ for all iterations $t$. Establishing the global convergence result remains as future work.

\bibliographystyle{apalike} 
\bibliography{grouse_noise}

\begin{thebibliography}{26}
\providecommand{\natexlab}[1]{#1}
\providecommand{\url}[1]{\texttt{#1}}
\expandafter\ifx\csname urlstyle\endcsname\relax
  \providecommand{\doi}[1]{doi: #1}\else
  \providecommand{\doi}{doi: \begingroup \urlstyle{rm}\Url}\fi

\bibitem[Absil et~al.(2009)Absil, Mahony, and Sepulchre]{absil2009optimization}
P-A Absil, Robert Mahony, and Rodolphe Sepulchre.
\newblock \emph{Optimization algorithms on matrix manifolds}.
\newblock Princeton University Press, 2009.

\bibitem[Armentano et~al.(2014)Armentano, Beltr{\'a}n, and
  Shub]{armentano2014average}
Diego Armentano, Carlos Beltr{\'a}n, and Michael Shub.
\newblock Average polynomial time for eigenvector computations.
\newblock \emph{arXiv preprint arXiv:1410.2179}, 2014.

\bibitem[Arora et~al.(2013)Arora, Cotter, and Srebro]{arora2013stochastic}
Raman Arora, Andy Cotter, and Nati Srebro.
\newblock Stochastic optimization of {PCA} with capped {MSG}.
\newblock In \emph{Advances in Neural Information Processing Systems}, pages
  1815--1823, 2013.

\bibitem[Balsubramani et~al.(2013)Balsubramani, Dasgupta, and
  Freund]{balsubramani2013fast}
Akshay Balsubramani, Sanjoy Dasgupta, and Yoav Freund.
\newblock The fast convergence of incremental {PCA}.
\newblock In \emph{Advances in Neural Information Processing Systems}, pages
  3174--3182, 2013.

\bibitem[Balzano and Wright(2014)]{balzano2014local}
Laura Balzano and Stephen~J Wright.
\newblock Local convergence of an algorithm for subspace identification from
  partial data.
\newblock \emph{Foundations of Computational Mathematics}, pages 1--36, 2014.

\bibitem[Balzano et~al.(2010)Balzano, Nowak, and Recht]{balzano2010online}
Laura Balzano, Robert Nowak, and Benjamin Recht.
\newblock Online identification and tracking of subspaces from highly
  incomplete information.
\newblock In \emph{Communication, Control, and Computing (Allerton), 2010 48th
  Annual Allerton Conference on}, pages 704--711. IEEE, 2010.

\bibitem[Balzano(2012)]{balzano2012handling}
Laura~Kathryn Balzano.
\newblock \emph{Handling missing data in high-dimensional subspace modeling}.
\newblock PhD thesis, University of Wisconsin -- Madison, 2012.

\bibitem[Bhojanapalli et~al.(2015)Bhojanapalli, Kyrillidis, and
  Sanghavi]{bhojanapalli2015dropping}
Srinadh Bhojanapalli, Anastasios Kyrillidis, and Sujay Sanghavi.
\newblock Dropping convexity for faster semi-definite optimization.
\newblock \emph{arXiv preprint arXiv:1509.03917}, 2015.

\bibitem[Brooks et~al.(2013)Brooks, Dul{\'a}, and Boone]{brooks2013pure}
J~Paul Brooks, JH~Dul{\'a}, and Edward~L Boone.
\newblock A pure l1-norm principal component analysis.
\newblock \emph{Computational statistics \& data analysis}, 61:\penalty0
  83--98, 2013.

\bibitem[Cand{\`e}s et~al.(2011)Cand{\`e}s, Li, Ma, and
  Wright]{candes2011robust}
Emmanuel~J Cand{\`e}s, Xiaodong Li, Yi~Ma, and John Wright.
\newblock Robust principal component analysis?
\newblock \emph{Journal of the ACM (JACM)}, 58\penalty0 (3):\penalty0 11, 2011.

\bibitem[Chen and Wainwright(2015)]{chen2015fast}
Yudong Chen and Martin~J Wainwright.
\newblock Fast low-rank estimation by projected gradient descent: General
  statistical and algorithmic guarantees.
\newblock \emph{arXiv preprint arXiv:1509.03025}, 2015.

\bibitem[d'Aspremont et~al.(2008)d'Aspremont, Bach, and Ghaoui]{d2008optimal}
Alexandre d'Aspremont, Francis Bach, and Laurent~El Ghaoui.
\newblock Optimal solutions for sparse principal component analysis.
\newblock \emph{The Journal of Machine Learning Research}, 9:\penalty0
  1269--1294, 2008.

\bibitem[De~Sa et~al.(2014)De~Sa, Olukotun, and R{\'e}]{de2014global}
Christopher De~Sa, Kunle Olukotun, and Christopher R{\'e}.
\newblock Global convergence of stochastic gradient descent for some nonconvex
  matrix problems.
\newblock \emph{arXiv preprint arXiv:1411.1134}, 2014.

\bibitem[Edelman et~al.(1998)Edelman, Arias, and Smith]{edelman1998geometry}
Alan Edelman, Tom{\'a}s~A Arias, and Steven~T Smith.
\newblock The geometry of algorithms with orthogonality constraints.
\newblock \emph{SIAM journal on Matrix Analysis and Applications}, 20\penalty0
  (2):\penalty0 303--353, 1998.

\bibitem[Golub and Van~Loan(2012)]{golub2012matrix}
Gene~H Golub and Charles~F Van~Loan.
\newblock \emph{Matrix computations}.
\newblock JHU Press, 4 edition, 2012.

\bibitem[He et~al.(2012)He, Balzano, and Szlam]{he2012incremental}
Jun He, Laura Balzano, and Arthur Szlam.
\newblock Incremental gradient on the grassmannian for online foreground and
  background separation in subsampled video.
\newblock In \emph{IEEE CVPR}, June 2012.

\bibitem[Jain et~al.(2013)Jain, Netrapalli, and Sanghavi]{jain2013low}
Prateek Jain, Praneeth Netrapalli, and Sujay Sanghavi.
\newblock Low-rank matrix completion using alternating minimization.
\newblock In \emph{Proceedings of the forty-fifth annual ACM symposium on
  Theory of computing}, pages 665--674. ACM, 2013.

\bibitem[Keshavan et~al.(2010)Keshavan, Montanari, and Oh]{keshavan2010matrix}
Raghunandan~H Keshavan, Andrea Montanari, and Sewoong Oh.
\newblock Matrix completion from a few entries.
\newblock \emph{Information Theory, IEEE Transactions on}, 56\penalty0
  (6):\penalty0 2980--2998, 2010.

\bibitem[Ngo and Saad(2012)]{ngo2012scaled}
Thanh Ngo and Yousef Saad.
\newblock Scaled gradients on grassmann manifolds for matrix completion.
\newblock In \emph{Advances in Neural Information Processing Systems}, pages
  1412--1420, 2012.

\bibitem[Nguyen et~al.(2014)Nguyen, Vu, et~al.]{nguyen2014random}
Hoi~H Nguyen, Van Vu, et~al.
\newblock Random matrices: Law of the determinant.
\newblock \emph{The Annals of Probability}, 42\penalty0 (1):\penalty0 146--167,
  2014.

\bibitem[Recht et~al.(2010)Recht, Fazel, and Parrilo]{recht2010guaranteed}
Benjamin Recht, Maryam Fazel, and Pablo~A Parrilo.
\newblock Guaranteed minimum-rank solutions of linear matrix equations via
  nuclear norm minimization.
\newblock \emph{SIAM review}, 52\penalty0 (3):\penalty0 471--501, 2010.

\bibitem[R.H.Keshavan(2012)]{RH2012}
R.H.Keshavan.
\newblock \emph{Efficient algorithms for collaborative filtering}.
\newblock PhD thesis, Stanford University, 2012.

\bibitem[Richt{\'a}rik and Tak{\'a}{\v{c}}(2014)]{richtarik2014iteration}
Peter Richt{\'a}rik and Martin Tak{\'a}{\v{c}}.
\newblock Iteration complexity of randomized block-coordinate descent methods
  for minimizing a composite function.
\newblock \emph{Mathematical Programming}, 144\penalty0 (1-2):\penalty0 1--38,
  2014.

\bibitem[Stewart and Sun(1990)]{stewart1990matrix}
Gilbert~W Stewart and Ji-guang Sun.
\newblock \emph{Matrix perturbation theory}.
\newblock Academic press, 1990.

\bibitem[Xu et~al.(2010)Xu, Caramanis, and Sanghavi]{xu2010robust}
Huan Xu, Constantine Caramanis, and Sujay Sanghavi.
\newblock Robust {PCA} via outlier pursuit.
\newblock In \emph{Advances in Neural Information Processing Systems}, pages
  2496--2504, 2010.

\bibitem[Zheng and Lafferty(2015)]{zheng2015convergent}
Qinqing Zheng and John Lafferty.
\newblock A convergent gradient descent algorithm for rank minimization and
  semidefinite programming from random linear measurements.
\newblock In \emph{Advances in Neural Information Processing Systems}, pages
  109--117, 2015.

\end{thebibliography}

\newpage
\onecolumn
\appendix
\section*{Appendix} 
\label{sec:appendix}
We first call out the following definitions that will be used frequently throughout the whole proof. For convenience, we will drop the subscript of all terms except $\epsilon_t$ and and $\zeta_t$ hereafter.  
\begin{definition}
	Let $v_{\parallel}, v_{\perp}$ denote the projection and residual of $v$ onto $R(U)$, \ie $v_{\parallel} = UU^T v$ and $v_{\perp} = v - v_{\parallel} = (I - UU^T)v$, and similarly let $\xi_{\parallel} = UU^T\xi$ and $\xi_{\perp} = \xi - \xi_{\parallel} = (I - UU^T)\xi$. It follows that, 
\begin{equation}
   p = v_{\parallel} + \xi_{\parallel} \qquad \text{and}\qquad r = v_{\perp} + \xi_{\perp}  \nn
\end{equation}
Define $w_1$ and $w_2$ as 
\begin{equation}
	w_1 =  U^T  v = U^T v_{\parallel}\qquad \text{and}\qquad w_2 = U^T\xi = U^T \xi_{\parallel} \nn
\end{equation}
\label{defn:paraperp}
\end{definition}

\section{Preliminaries} 
\label{sec:appendix_a}
We start by providing the following lemma that we will use regularly in the manipulation of the matrix $\bar{U}^T U$, which is relevant for both our metrics of discrepancy and similarity between the subspaces. The proof can be found in \cite{stewart1990matrix}. 
\begin{lemma}[\cite{stewart1990matrix}, Theorem 5.2]
  There are unitary matrices $Q$, $\bar{Y}$, and $Y$ such that 
\begin{equation}
	Q\bar{U}\bar{Y} := \bordermatrix{&d\cr
                d& I \cr
               d & 0 \cr
                n-2d & 0 } = \Lambda_1, 
	\quad
	Q U Y := \bordermatrix{&d\cr
                d& \Gamma \cr
               d & \Sigma \cr
                n-2d & 0 } = \Lambda_2,   \nonumber
\end{equation}
where $\Gamma = \diag{(\cos\phi_{t,1}, \dots, \cos\phi_{t,d})}, \Sigma = \diag{(\sin\phi_{t,1}, \dots, \sin\phi_{t,d})}$ with $\phi_{t, i}$ being the $i^{th}$ principal angle between $R(U)$ and $R(\bar U)$ defined in Definition \ref{defn:angles}.
\label{lem:subsprelate}
\end{lemma}
In words, there are unitary matrices $Q$, $\bar{Y}$, and $Y$ such that $\bar U = Q^T \Lambda_1 \bar Y^T$ and $U =Q^T \Lambda_2 Y^T$. Letting $B = \bar U^T UU^T\bar U$, we call out the following simplified quantities for future reference :
\begin{equation}
 	B = \bar Y \Gamma^2  \bar Y^T  \quad \text{and} \quad B^T B = \bar Y \Gamma^{4} \bar Y^T \;. 
 	\label{subrelate3}
 \end{equation}

Next we present the following two lemmas that are for us to relate the projection ($v_{\parallel}$) and residual ($v_{\perp}$) to both of our metrics $\epsilon_t$ and $\zeta_t$. The proofs can be found in Appendix \ref{sec:proof_of_concentration_results}. 
\begin{lemma}[\cite{balzano2014local}, Lemma 2.12]
	Given any matrix $Q\in \R^{d\times d}$ suppose that $x \in \R^{d}$ is a random vector with entries identically distributed, zero-mean, and uncorrelated, then 
	\begin{equation}
		\mathbb{E}\left[\frac{x^T Q x}{x^T x}\right] = \frac{1}{d}\trace(Q) \nonumber
	\end{equation}
	\label{lem:expisovec}
\end{lemma}

\begin{lemma}
Let $X = [X_1, \cdots, X_d]$ with $X_i \in [0,1], i = 1, \dots, d$, then 
\begin{align}
	 d - \sum_{i = 1}^{d}X_i \geq  1 - \Pi_{i = 1}^{d}X_i \nn 
\end{align}
Given $\Pi_{i = 1}^{d} X_i \geq \frac{1}{2}$, we have
\begin{equation}
   d - \sum_{i = 1}^{d}X_i \leq 2\left(1 - \Pi_{i = 1}^{d}X_i\right) \nn
\end{equation}
	\label{lem:zeta_eps_relate}
\end{lemma}

Now we are ready to present the following lemma that shows a relationship relating the projection ($v_{\parallel}$) and residual ($v_{\perp}$) to our metrics $\epsilon_t$ and $\zeta_t$ in terms of expectation. This is a central result used to obtain the expected convergence rates of both $\epsilon_t$ and $\zeta_t$. 
\begin{lemma}
Under Definition \ref{defn:paraperp}, we have the following
\begin{subequations}
	\begin{align}
		&\mathbb{E}\left[\frac{\|v_{\perp}\|^2}{\|v\|^2}\bigg\lvert U\right] = \frac{\epsilon_t}{d} \label{ineq1:vconc}\\
		&\mathbb{E}\left[\frac{\|v_{\perp}\|^2}{\|v\|^2}\bigg\lvert U\right] \geq \frac{1 - \zeta_t}{d} \label{ineq2:vconc} \\
		&\mathbb{E}\left[\frac{\left\|(I - \bar U\bar U^T)v_{\parallel}\right\|^2}{\|v\|^2}\bigg\lvert U\right] \geq \frac{\cos^2\phi_{d}}{d}\epsilon_t \label{ineq3:vconc}
	\end{align}
\end{subequations}
where $\phi_{d}$ is the largest principal angle between our estimated subspace $R(U_t)$ and the true subspace $R(\bar U)$ (Definition \ref{defn:angles}). 
	\label{lem:vconc}
\end{lemma}
\begin{proof}
	According to Lemma \ref{lem:expisovec}, we have 
	\begin{align}
		\mathbb{E}\left[\frac{\|v_{\perp}\|^2}{\|v\|^2} \bigg\lvert U\right] = \mathbb{E}\left[\frac{\|\bar U s\|^2 - \|UU^T\bar U s\|^2}{\|\bar U s\|^2}\bigg\lvert U\right] &\overset{\vartheta_1}= \mathbb{E}\left[\frac{s^T\bar Y(I - \Gamma^2)\bar Y^T s}{s^Ts}\bigg\lvert U\right] \nn \\
		&\overset{\vartheta_2}= \frac{1}{d}\trace\left(I - \Gamma^2\right) \nn \\
		&= \epsilon_t / d \nn 
	\end{align}
	where $\vartheta_1$ follows by $\|\bar U s\|^2 = \|s\|^2$ and $\vartheta_2$ follows from Lemma \ref{lem:expisovec}. 

	(\ref{ineq2:vconc}) is a direct result of Lemma \ref{lem:zeta_eps_relate} by setting $X = \textbf{diag}(\Gamma^2)$, \ie 
	\begin{equation}
		\mathbb{E}\left[\frac{\|v_{\perp}\|^2}{\|v\|^2}\bigg\lvert U\right] = \frac{1}{d}\trace(I - \Gamma^2) \geq \frac{1 - \zeta_t}{d}  \nn
 	\end{equation}

   Finally for (\ref{ineq3:vconc}), we have
    \begin{align}
    	\mathbb{E}\left[\frac{\left\|(I - \bar U\bar U^T)v_{\parallel}\right\|^2}{\|v\|^2}\bigg\lvert U\right] = \mathbb{E}\left[\frac{s^T\bar Y \Gamma^2(I - \Gamma^2) \bar Y^T s}{s^Ts}\bigg\lvert U\right] &= \frac{1}{d}\trace\left(\Gamma^2(I - \Gamma^2)\right) \nn \\
    	&\geq \cos^2\phi_{d}\trace\left(I - \Gamma^2\right)\cdot \frac{1}{d} \nn \\
    	&= \frac{\cos^2\phi_d}{d}\epsilon_t \nn 
    \end{align}
\end{proof}

\section{Proof of Main Results} 
\label{sec:proof_of_main_results}
\subsection{Noiseless Case}
In this section, given the results presented in Section \ref{sec:supportingresults} and Appendix \ref{sec:appendix_a}, we provide the proofs of our main results (Theorem \ref{thm:detconvgrate}, \ref{thm:localepsconvg}). 
\paragraph{Proof of Theorem \ref{thm:detconvgrate}}{
	\begin{proof}
		According to Lemma \ref{lem:mono}, $\zeta_t$ is an non-decreasing sequence. Therefore, there exists $T \geq 1$ such that $\zeta_t \leq 1 - \frac{\rho'}{d}$, $\forall t\leq T$ where $\rho' \in (0,1]$. Then Lemma \ref{lem:mono} together with Lemma \ref{lem:vconc} yield the following, for any $t \leq T$,
		\begin{align}
			\mathbb{E}_{s_t}\left[\frac{\zeta_{t + 1}}{\zeta_t} \bigg\lvert U\right] \geq 1 + \mathbb{E}_{s_t}\left[\frac{\|v_{\perp}\|^2}{\|v\|^2}\bigg \lvert U\right] &\geq 1 + \frac{1 - \zeta_t}{d} \geq  1 + \frac{\rho'}{d^2}\;.
		\end{align}
		It follows that
		\begin{equation}
			\mathbb{E}\left[\zeta_{t + 1} \big\lvert U\right] \geq \left(1 + \frac{\rho'}{d^2}\right) \zeta_t
		\end{equation}
		Taking expectation of both sides, we obtain the following
		\begin{equation}
			\mathbb{E}\left[\zeta_{t + 1}\right] \geq \left(1 + \frac{\rho'}{d^2}\right)\mathbb{E}[\zeta_t]
		\end{equation}
		Therefore after $K_1 \geq (d^2 / \rho' + 1)\log\frac{1 - \frac{\rho'}{2}}{\mathbb{E}[\zeta_0]}$ iterations of GROUSE we have 
		\begin{equation}
			\mathbb{E}\left[\zeta_{K_1}\right] \geq \left(1 + \frac{\rho'}{d^2}\right)^{K_1} \mathbb{E}[\zeta_0] \geq \left(\left(1 + \frac{\rho'}{d^2}\right)^{\frac{d^2}{\rho'} + 1}\right)^{\log\frac{1 - \frac{\rho'}{2}}{\mathbb{E}[\zeta_0]}} \mathbb{E}[\zeta_0] \geq \mathbb{E}[\zeta_0] e^{\log\frac{1 - \frac{\rho'}{2}}{\mathbb{E}[\zeta_0]}} = 1 - \frac{\rho'}{2}  \nn
		\end{equation}
		Therefore, 
		\begin{equation}
		    \mathbb{P}\left(\zeta_{K_1} \geq \frac{1}{2}\right) = 1 - \mathbb{P}\left(1 - \zeta_{K_1} \geq \frac{1}{2}\right) \overset{\vartheta_1}\geq 1 - \frac{\mathbb{E}[1 - \zeta_{K_1}]}{1/2} \geq 1 - \rho'  
		    \label{pf:K1_prob}
		\end{equation}
		where $\vartheta_1$ follows by applying Markov inequality to the nonnegative random variable $1 - \zeta_{K_1}$. If $T \leq K_1$, then (\ref{pf:K1_prob}) automatically holds. 
		Therefore, together with the following derived from Lemma \ref{lem:exp_zeta0}, we obtain $\log\left(\frac{1 - \rho'/2}{\mathbb{E}[\zeta_0]}\right) = \log\left(\frac{1 - \rho'/2}{C\left(\frac{d}{n e}\right)^d}\right) = \mu_0 d\log(n)$ where $\mu_0 = 1 + \frac{\log \frac{(1 - \rho')}{C} + d\log (e/d)}{d\log n}$.
	\end{proof}
}

\paragraph{Proof of Theorem \ref{thm:localepsconvg}}{
\begin{proof}
This proof follows the same reasoning as the proof of [Theorem 1, \cite{richtarik2014iteration}].

Given Lemma \ref{lem:mono} we have 
\begin{equation}
	\mathbb{E}\left[\epsilon_{t + 1}\big\lvert U\right] \overset{\vartheta_1}{\leq} \epsilon_t - \mathbb{E}\left(\frac{\|(I - UU^T)v_{\parallel}\|^2}{\|v_{\parallel}\|^2}\bigg\lvert U\right) \leq \epsilon_t - \mathbb{E}\left(\frac{\|(I - UU^T)v_{\parallel}\|^2}{\|v\|^2}\bigg\lvert U\right)\overset{\vartheta_2}\leq \left(1 - \frac{\cos^2\phi_{t,d}}{d}\right)\epsilon_t \nn 
\end{equation}
where $\vartheta_1$ holds since $UU^T$ and $I - UU^T$ are orthogonal complement of each other; and $\vartheta_2$ follows from Lemma \ref{lem:vconc}. 
Also note that Lemma \ref{lem:mono} implies that $\cos^{2}\phi_{t,d} \geq \zeta_{t} \geq \zeta_k \geq 1/2, \forall t \geq k$, therefore
$$\E\left[\epsilon_{t+1}\big\lvert U\right] \leq \epsilon_t \left( 1 - 1/2 d\right)\;.$$ 
Taking expectations of both sides and considering $\epsilon_t$ at $t={k+K_2}$ with $K_2\geq 2 d \log\left(1/\epsilon^{\ast}\rho\right)$, we have
\begin{align}
\E[\epsilon_{k+K_2}] \leq \epsilon_k\left(1 - \frac{1}{2 d}\right)^{K_2} \overset{\vartheta_3}\leq  \mathbb{E}\left[\epsilon_k\right] \left(1 - \frac{1}{2 d}\right)^{2 d\log\left(\mathbb{E}[\epsilon_k]/\epsilon^* \rho\right)} \leq \mathbb{E}[\epsilon_k] e^{-\log(\mathbb{E}[\epsilon_k] / \epsilon^* \rho)} = \epsilon^* \rho \nn
\end{align}
where $\vartheta_3$ follows by Lemma \ref{lem:zeta_eps_relate}, \ie given $\zeta_k \geq 1/2$ we have $\epsilon_k \leq 1$. Again using the Markov inequality for the nonnegative random variable $\epsilon_t$ we have $$\bP(\epsilon_{k+K_2} \leq \epsilon^*) = 1 - \bP(\epsilon_{k+K_2} \geq \epsilon^*) \geq 1-\frac{\E[\epsilon_{k+K_2}]}{\epsilon^*} \geq 1- \rho\;.$$
\end{proof}
}

\subsection{Noisy Case}
\label{subsec:eps_proof}
Before we prove the main results (Theorem \ref{thm:detratio_exp},\ref{thm:eps_diff_exp}) of the noisy case, we first call out the following lemmas that will be used frequently throughout the proofs in this section. The proofs are provided in Appendix \ref{sec:proof_of_concentration_results}. 
\begin{lemma}
Given Condition \ref{cond:fullnoise} is satisfied, we have 
	\begin{align}
		& \mathbb{E}\left[\|\xi_{\perp}\|^2\big\lvert v\right] \leq (1 - d/n)\sigma^2 \|v\|^2 &\mathbb{E}\left[\|\xi_{\parallel}\|^2\big\lvert v\right] \leq \frac{d}{n}\sigma^2 \|v\|^2  \nn \\
		& \mathbb{E}\left[\|p\|^2\big\lvert v\right] \leq \left(1 + \frac{d}{n}\sigma^2\right)\|v\|^2  &\mathbb{E}\left[\|r\|^2\big\lvert v\right]\leq \mathbb{E}\left[\|v_{\perp}\|^2 \big\lvert v\right] + \left(1 - \frac{d}{n}\right)\sigma^2\|v\|^2 \nn
	\end{align}
	\label{lem:rp_conc}
\end{lemma}

\begin{lemma}
Letting $\Delta_1 = v_{\perp}^T\xi_{\perp} + w_2^T(\bar U^TU_t)^{-1}\bar U^Tr$ and $\Delta_2 = \xi_{\parallel}^T(I - \bar U\bar U^T)(v_{\parallel} - \xi_{\perp})$, we have 
	\begin{equation}
		\mathbb{E}\left[\Delta_1 \big\lvert U\right] = 0 \qquad \qquad \mathbb{E}\left[\Delta_2 \big\lvert U\right] = 0    \nn 
	\end{equation}
	\label{lem:exp_delta}
\end{lemma}
\begin{proof}
	The proof can be found in Appendix \ref{sec:proof_of_concentration_results}. 
\end{proof}

Now we are ready to prove Theorem \ref{thm:detratio_exp} and Theorem \ref{thm:eps_diff_exp}. 
\paragraph{Proof of Theorem \ref{thm:detratio_exp}}{
\begin{proof}
Given Lemma \ref{lem:det_ratio_exp} and that $1 - \alpha = \frac{\|v_{\perp}\|^2}{\|r\|^2}$,  we have 
	\begin{align}
		\mathbb{E}\left[(1 - \alpha)^2\frac{\|r\|^2}{\|p\|^2}\bigg\lvert U\right] = \mathbb{E}\left[\frac{\|v_{\perp}\|^2}{\|r\|^2}\frac{\|v_{\perp}\|^2}{\|p\|^2}\bigg\lvert U\right] &\overset{\vartheta_1}\geq \mathbb{E}_{v}\left[\frac{\|v_{\perp}\|^2}{\mathbb{E}\left(\|r\|^2 \big\lvert v\right)}\frac{\|v_{\perp}\|^2}{\mathbb{E}\left(\|p\|^2\big\lvert v\right)}\bigg\lvert U\right] \nn \\
		&\overset{\vartheta_2}\geq  \frac{1}{1 + \frac{d}{n}\sigma^2}\mathbb{E}_{v}\left[\frac{\|v_{\perp}\|^2 / \|v\|^2}{\|v_{\perp}\|^2 / \|v\|^2 + (1 - d/n)\sigma^2}\frac{\|v_{\perp}\|^2}{\|v\|^2}\bigg\lvert U\right] \nn \\
		&\overset{\vartheta_3}\geq \frac{1}{1 + \frac{d}{n}\sigma^2}\mathbb{E}\left[\frac{\|v_{\perp}\|^2}{\|v\|^2}\bigg\lvert U\right] \left(1 - \frac{\sigma^2}{\mathbb{E}\left[\frac{\|v_{\perp}\|^2}{\|v\|^2}\big\lvert U\right] + \sigma^2}\right) \nn \\
		&\overset{\vartheta_4}\geq \frac{1}{\left(1 + \frac{d}{n}\sigma^2\right)}\frac{1 - \zeta_t}{d}\left(1 - \frac{\sigma^2}{(1 - \zeta_t)/d + \sigma^2}\right) \nn
	\end{align}
	where $\vartheta_1$ follows by the fact that $\xi_{\perp}$ and $\xi_{\parallel}$ are independent of each other,  and for each iteration $t$ given $U$ and $v$, $v_{\parallel}$ and $v_{\perp}$ are fixed, then applying Jensen's Inequality yields the results; $\vartheta_2$ follows from Lemma \ref{lem:rp_conc}; again we apply Jensen's Inequality in terms of $\|v_{\perp}\|^2 / \|v\|^2$ for $\vartheta_3$; and $\vartheta_4$ follows from Lemma \ref{lem:vconc}.
\end{proof}
}

\paragraph{Proof of Theorem \ref{thm:eps_diff_exp}}{\begin{proof}
According to Lemma \ref{lem:epsdiff_expr} we have the following,
	\begin{align}
		\mathbb{E}\left[\epsilon_t - \epsilon_{t + 1} \big\lvert U\right] &\overset{\vartheta_1}\geq \mathbb{E}_{v}\left[\mathbb{E}\left(\frac{\|(I - \bar U\bar U^T)p\|^2}{\|p\|^2} - \frac{\|(I - \bar U\bar U^T)\left(\xi - \alpha r\right)\|^2}{\|p\|^2} \bigg\lvert v,U\right)\right] \nn \\
		&\overset{\vartheta_2}= \mathbb{E}_v\left[\mathbb{E}\left(\frac{\|(I - \bar U\bar U^T)v_{\parallel}\|^2 - \|(I - \bar U\bar U^T)(\xi_{\perp} - \alpha r)\|^2 + 2(1 - \alpha)\Delta_1}{\|p\|^2} \bigg\lvert v,U\right)\right] \nn \\ 
		&\overset{\vartheta_3}\geq \mathbb{E}_v\left[\frac{\|(I - \bar U\bar U^T)v_{\parallel}\|^2 - \mathbb{E}\left[\|\xi_{\perp} - \alpha r\|^2 \big\lvert v, U\right]}{\mathbb{E}\left[\|p\|^2 \big\lvert v, U\right]}\right] \nn \\
		&\overset{\vartheta_4}\geq \frac{1}{1 + \frac{d}{n}\sigma^2} \mathbb{E}_{v}\left[\frac{\|(I - \bar U\bar U^T)v_{\parallel}\|^2}{\|v\|^2} - \frac{\|v_{\perp}\|^2}{\|v\|^2}\frac{(1 - d/n)\sigma^2}{\|v_{\perp}\|^2/\|v\|^2 + (1 - d/n)\sigma^2}\bigg\lvert U\right] \nonumber \\
		&\overset{\vartheta_5}\geq \frac{1}{1 + \frac{d}{n}\sigma^2} \left(\mathbb{E}_{v}\left[\frac{\|(I - \bar U\bar U^T)v_{\parallel}\|^2}{\|v\|^2} \bigg\lvert U\right] - \mathbb{E}\left[\frac{\|v_{\perp}\|^2}{\|v\|^2}\bigg\lvert U\right] \frac{(1 - d/n)\sigma^2}{\mathbb{E}\left[\|v_{\perp}\|^2/\|v\|^2 \big\lvert U\right] + (1 - d/n)\sigma^2}\right) \nonumber \\
		&\overset{\vartheta_6}\geq \frac{1}{1 + \frac{d}{n}\sigma^2}\left(\frac{\cos^2\phi_d}{d}\epsilon_t - \left(1 - \frac{d}{n}\right)\sigma^2\frac{\epsilon_t / d}{\epsilon_t / d + (1 - d/n)\sigma^2}\right) \nn 
	\end{align}
	where $\vartheta_1$ holds since $1 - \|\bar U\bar U^T p\|^2 / \|p\|^2 = \|(I - \bar U\bar U^T)p\|^2 / \|p\|^2$ and $\|p + (1 - \alpha) r\|^2 = \|p\|^2 + (1 - \alpha)^2\|r\|^2 \geq \|p\|^2$; $\vartheta_2$ follows by 
	\begin{equation}
	 	(1 - \alpha)\Delta_1 = \xi_{\parallel}^T(I - \bar U\bar U^T)\left(v_{\parallel} - \xi_{\perp} + \alpha r\right) = (1 - \alpha)\xi_{\parallel}^T(I - \bar U\bar U^T)(v_{\parallel} - \xi_{\perp}) \nn
	 \end{equation} 
	 $\vartheta_3$ follows by Lemma \ref{lem:exp_delta} and the fact given $v$, $p$ and $\xi_{\perp} - \alpha r$ are independent, then applying Jensen's inequality; and $\vartheta_4$ holds due to Lemma \ref{lem:rp_conc} and the following:
	  $$\alpha = 1 - \frac{\|v_{\perp}\|^2}{\|r\|^2} = \frac{\|\xi_{\perp}\|^2 + 2\xi_{\perp}^Tv_{\perp}}{\|r\|^2}$$
	  which implies
	  \begin{align}
	  	\mathbb{E}\left[\frac{\|\xi_{\perp} - \alpha r\|^2}{\|v\|^2}\bigg\lvert v, U\right] = \mathbb{E}\left[\frac{\|\xi_{\perp}\|^2 - \alpha\left(2\xi_{\perp}^Tr - \alpha\|r\|^2\right)}{\|v\|^2}\bigg\lvert v, U\right] & = \mathbb{E}\left[\frac{(1 - \alpha)\|\xi_{\perp}\|^2}{\|v\|^2}\bigg\lvert v, U\right] \nn \\
	  	&\overset{\vartheta_7}\leq \frac{\|v_{\perp}\|^2}{\|v\|^2} \frac{\mathbb{E}\left[\|\xi_{\perp}\|^2\big\lvert v, U\right]}{\mathbb{E}\left[\|r\|^2\big\lvert v, U\right]} \nn \\
	  	&=  \frac{\|v_{\perp}\|^2}{\|v\|^2}\frac{(1 - d/n)\sigma^2\|v\|^2}{\|v_{\perp}\|^2 + (1 - d/n)\sigma^2\|v\|^2} \nn 
	  \end{align}
	  where in $\vartheta_7$ we used Jensen's Inequality in terms of $\xi_{\perp}$. 

	  Finally, for $\vartheta_5$ we again apply Jensen's Inequality and then use Lemma \ref{lem:vconc} for $\vartheta_6$ to complete the proof.
\end{proof}}

\section{Proof of Supporting Theory} 
\label{sec:proof_of_supporting_theory}

\paragraph{Proof of Lemma \ref{lem:epsdiff_expr}}{\begin{proof}
  We first rotate $U_t$ and $U_{t + 1}$ via an orthogonal matrix defined as 
\begin{equation}
W_t := \left[\left. \frac{w}{\|w\|}  \, \right| \, Z \right] \nn
\end{equation}
where $Z$ is an $d \times (d-1)$ matrix with orthonormal columns whose columns span $N(w^T)$, where $N(\cdot)$ denotes the nullspace and so this gives the orthogonal complement of $w$. Note that $R(U_t)$ is unchanged by multiplying $U_t$ with $W_t$, i.e, $R(U_t) = R(U_t W)$. 
This allows us to equivalently write the update equation as follows, 
\begin{eqnarray}
U_{t + 1} W = U_t W + \left[ \frac{y}{\| y \|} - \frac{p}{\| p \|} \right] \frac{w^T}{\|w\|} W = U_t W + \left[ \frac{y}{\| y \|} - \frac{p}{\| p \|} \right] \left[ \begin{matrix} 1 & 0 & 0 & \ldots & 0 \end{matrix} \right] \nn
\end{eqnarray}
Since Frobenius norm is invariant under orthogonal transformation, it directly implies the following: 
 \begin{equation}
    \epsilon_{t} - \epsilon_{t + 1} = \left\|\bar U^T U_{t + 1}W\right\|_F^2 - \left\|\bar U^T U_t W\right\|_F^2 = \frac{\left\|\bar U\bar U^Ty\right\|^2}{\left\|y\right\|^2} - \frac{\|\bar U\bar U^T p\|^2}{\|p\|^2}  \label{eq:eps_diff}
  \end{equation}
Note that $I - \bar U\bar U^T$ is the orthogonal complement of $\bar U\bar U^T$, and $(I - \bar U\bar U^T)v = 0$, hence we have the following, 
\begin{align}
	\frac{\|\bar U\bar U^T y\|^2}{\|y\|^2} = \frac{\|\bar U\bar U^T(v + \xi - \alpha r)\|^2}{\|v + (\xi - \alpha r)\|^2} &= 1 - \frac{\|(I - \bar U\bar U^T)(\xi - \alpha r)\|^2}{\|v + \xi - \alpha r\|^2} \nn
\end{align}
\end{proof}}
\begin{remark}[Proof of the Second Claim in Lemma \ref{lem:mono}]
	Note that when there is no noise, $y = v \in R(\bar U)$ and $p = v_{\parallel}$, therefore (\ref{eq:eps_diff}) implies the following
	\begin{equation}
	  	\frac{\|\bar U\bar U^T y\|^2}{\|y\|^2} - \frac{\|\bar U\bar U^T p\|^2}{\|p\|^2} = 1 - \frac{\|\bar U\bar U^T v_{\parallel}\|^2}{\|v_{\parallel}\|^2}
	  \end{equation}  
    this completes the proof of the second claim of Lemma \ref{lem:mono}. 
\end{remark}

\paragraph{Proof of Lemma \ref{lem:det_ratio_exp}}{
\begin{proof}
Since $\det\left(A + a b^{T}\right) = \det(A)\left(1 + b^T A^{-1} a\right)$, we obtain
\begin{align}
    \det \bar U^T U_{t + 1} &= \det \left(\bar U^T U_t + \left(\frac{\bar U^T y}{\|y\|} - \frac{\bar U^T p}{\|p\|}\right)\frac{w^T}{\|w\|}\right) \nn \\
    &= \det (\bar U^T U_t ) \cdot \left(1 + \frac{w^T(\bar U^T U)^{-1}\bar U^Ty}{\|p\| \|p + (1 - \alpha)r\|} - \frac{w^T(\bar U^TU)^{-1}\bar U^T p}{\|p\|^2}\right) \label{eq:det_ratio}
  \end{align}
Note that 
\begin{align}
	& w^T(\bar U^TU)^{-1}p = w^T(\bar U^TU)^{-1}\bar U^T U w = \|w\|^2 \nn\\
	& w_1^T(\bar U^TU)^{-1}r = s^T\bar U^TU(\bar U^TU)^{-1}r = v^Tr = \|v_\perp\|^2 + v_{\perp} \xi_{\perp} \nn
\end{align}
together with the fact $\|p + (1 - \alpha)r\|^2 = \|p\|^2 + (1 - \alpha)^2\|r\|^2$ yields
\begin{align}
 	\mathbb{E}\left[\frac{\zeta_{t + 1}}{\zeta_t}\bigg\lvert U\right] &= \mathbb{E}\left[\left(\frac{\|p\|^2 + \left(1 - \alpha\right)\left(\|v_{\perp}\|^2 + v_{\perp}^T\xi_{\perp} + w_2^T(\bar U^TU_t)^{-1}\bar U^Tr\right)}{\left\|p + \left(1 - \alpha\right)r\right\|\|p\|}\right)^2\bigg\lvert U\right] \nn \\
	&\overset{\vartheta_1}= \mathbb{E}\left[\left(\frac{\|p\|^2 + \left(1 - \alpha\right)^2\|r\|^2 + (1 - \alpha)\Delta_2}{\left\|p + \left(1 - \alpha\right)r\right\|\|p\|}\right)^2\bigg\lvert U\right] \nn \\
	&\overset{\vartheta_2}\geq \mathbb{E}\left[\frac{\|p\|^2 + (1 - \alpha)^2\|r\|^2}{\|p\|^2} + 2(1 - \alpha)\frac{\Delta_2}{\|p\|^2} \bigg\lvert U\right] \nn \\
 	&\overset{\vartheta_3}= \mathbb{E}\left[1 + (1 - \alpha)^2\frac{\|r\|^2}{\|p\|^2}\bigg\lvert U\right]
 \end{align} 
where $\vartheta_1$ follows by $\|v_{\perp}\|^2 =(1 - \alpha)\|r\|^2$ and $\Delta_2 = v_{\perp}^T\xi_{\perp} + w_2^T(\bar U^TU_t)^{-1}\bar U^Tr$; $\vartheta_2$ follows from $p \perp r$, this implies $\|p + (1 - \alpha)r\|^2 = \|p\|^2 + (1 - \alpha)^2\|r\|^2$; and $\vartheta_3$ holds since according to Lemma \ref{lem:exp_delta} and our assumptions that $\|v\|^2 < \infty$ this implies that $\|v_{\perp}\|^2$ and with the Gaussian noise both $\|p\|^2$ and $\|r\|^2$ are bounded with probability equals $1$. Therefore, we have $\mathbb{E}\left[2(1 - \alpha)\Delta_1/\|p\|^2 \big\lvert U\right] = 0$. 
\end{proof} 
}

\begin{remark}[Proof of the First Claim in Lemma \ref{lem:mono}]
	For the noise-free case, $y = v$ and $w = U^T v = ^T\bar U s_t$, $w^T(\bar U^TU)^{-1}\bar U^Ty = s^T\bar U^T U(\bar U^TU)^{-1}\bar U^T v = \|v\|^2$. Hence (\ref{eq:det_ratio}) yields
	\begin{equation}
		\frac{\zeta_{t + 1}}{\zeta_t} = \frac{\|v\|^2}{\|v_{\parallel}\|^2} = 1 + \frac{\|v_{\perp}\|^2}{\|v_{\parallel}\|^2} \nonumber
	\end{equation}
	this completes the proof of the first claim of Lemma \ref{lem:mono}.
\end{remark}

\section{} 
\label{sec:proof_of_concentration_results}
\paragraph{Proof of Lemma \ref{lem:expisovec}}{\begin{proof}
This proof is identical to that of [Lemma 2.12, \cite{balzano2014local}], but we note that their assumption that $x$ be Gaussian is not necessary.
	\begin{eqnarray}
		\mathbb{E}\left[\frac{x^T Q x}{x^T x}\right] = \sum_{i\neq j}\mathbb{E}\left[\frac{x_i x_j}{\|x\|_2^2}\right]Q_{ij} + \sum_{i = 1}^{d}\mathbb{E}\left[\frac{x_i^2}{\|x\|_2^2}\right]Q_{ii} \overset{\xi_1}= \sum_{i = 1}^{d}\mathbb{E}\left[\frac{x_i^2}{\|x\|_2^2}\right] \overset{\xi_2}= \frac{1}{d}\trace(Q) \nonumber
	\end{eqnarray}
	where $\xi_1$ follows from the fact that $\mathbb{E}\left[\frac{x_ix_j}{\|x\|_2^2}\right] = 0$ and $\xi_2$ follows by 
	\begin{equation}
		1 = \mathbb{E}\left[\frac{\sum_{i = 1}^{d}x_i^2}{\sum_{i = 1}^{d}x_i^2}\right] = \sum_{i = 1}^{d}\mathbb{E}\left[\frac{x_i^2}{\sum_{i = 1}^{d}x_i^2}\right] = d \mathbb{E}\left[\frac{x_i^2}{\sum_{i = 1}^{d}x_i^2}\right], \quad i = 1, \dots,d \nonumber
	\end{equation}
	since each $x_i$ is identically distributed.
\end{proof}}

\paragraph{Proof of Lemma \ref{lem:zeta_eps_relate}}{
	\begin{proof}
		The proof of the first claim is similar to that of [Lemma 16, \cite{de2014global}]; we briefly sketch it here. Let $f_1(X) = d - 1 - \sum_{i = 1}^{d}X_i + \Pi_{i = 1}^{d}X_i$, then 
		$\frac{\partial f_1(X)}{\partial X_i} = -1 + \Pi_{j \neq i} X_i \leq 0$. That is, $f_1(X)$ is a decreasing function of each component. Hence
		\begin{equation}
			f_1(X) \geq f(\textbf{1}) = 0 \nn 
		\end{equation}
		For the second claim, let $f_2(X) = 2\left(1 - \Pi_{i = 1}^{d}X_i\right) - d + \sum_{i = 1}^{d}X_i$, then given $\Pi_{i = 1}^{d} X_i \geq 1/2$, we have 
	 \begin{equation}
	 	\frac{\partial f_2(X)}{\partial X_i} = -2\Pi_{j \neq i} X_i + 1 \leq 0 \Rightarrow f_2(X) \geq f_2(\textbf{1}) = 0 \nn 
	 \end{equation}
	\end{proof}
}

\paragraph{Proof of Lemma \ref{lem:rp_conc}}{\begin{proof}
	Note that $\|\xi\|^2$, $\|\xi_{\perp}\|^2$ and $\|\xi_{\parallel}\|^2$ are all $\chi^2$ distributed random variables with degrees $n$, $n - d$ and $d$. This implies the first two parts. 

	And for the last two inequalities, we have 
	\begin{align}
		&\mathbb{E}\left[\xi_{\parallel}^T v_{\parallel}\right] = \mathbb{E}_{v}\left[\mathbb{E}\left[\xi^TUU^Tv \big\lvert v\right]\right] = 0  \nn \\
		&\mathbb{E}\left[\xi_{\perp}^T v_{\perp}\right] = \mathbb{E}_{v}\left[\xi^T(I - UU^T)v \big \lvert v\right] = 0  \nn
	\end{align}
	which implies the following:
	\begin{align}
		&\mathbb{E}\left[\|r\|^2 \big\lvert v\right] = \mathbb{E}\left[\|v_{\perp}\|^2\lvert v\right] + \mathbb{E}\left[\|\xi_{\perp}\|^2\lvert v\right] \leq \mathbb{E}\left[\|v_{\perp}\|^2\lvert v\right] + \left(1 - \frac{d}{n}\right)\sigma^2\|v\|^2 \nn \\
		&\mathbb{E}\left[\|p\|^2 \big\lvert v\right] = \mathbb{E}\left[\|v_{\parallel}\|^2\lvert v\right] + \mathbb{E}\left[\|\xi_{\parallel}\|^2\lvert v\right] \leq \mathbb{E}\left[\|v_{\parallel}\|^2\lvert v\right] + \frac{d}{n}\sigma^2\|v\|^2 \leq \left(1 + \frac{d}{n}\sigma^2\right)\|v\|^2   \nn
	\end{align}
\end{proof}}

\paragraph{Proof of Lemma \ref{lem:exp_delta}}{\begin{proof}
	\begin{align}
		\mathbb{E}\left[\Delta_1\big\lvert U\right] &= \mathbb{E}_{v}\left[\mathbb{E}\left[v^T(I - UU^T)\xi_{\perp}\big\lvert U, v\right]\right] + \mathbb{E}_{v}\left[\mathbb{E}\left[\xi^TU(\bar U^TU)^{-1}\bar U^T v_{\perp}\big\lvert U, v\right]\right] \nn \\
		& \quad + \mathbb{E}\left[\xi^TU(\bar U^TU)^{-1}\bar U^T (I - UU^T)\xi\big\lvert U\right]\nn \\
		&\overset{\vartheta_1}= \trace\left(\mathbb{E}\left[\xi^TU(\bar U^TU)^{-1}\bar U^T (I - UU^T)\xi\big\lvert U\right]\right) \nn \\
		&= \trace\left((\bar U^TU)^{-1}\bar U^T (I - UU^T)\mathbb{E}\left[\xi\xi^T\right]U\right) \nn \\
		&\overset{\vartheta_2}= 0 \nn
	\end{align}
	where $\vartheta_1$ holds since the following, 
	\begin{align}
	 	& \mathbb{E}\left[v^T(I - UU^T)\xi_{\perp}\big\lvert U, v\right]] = \mathbb{E}\left[v^T(I - UU^T)\xi \big\lvert U, v\right] = 0 \nn \\
	 	& \mathbb{E}\left[\xi^TU(\bar U^TU)^{-1}\bar U^T v_{\perp}\big\lvert U, v\right] = \mathbb{E}\left[\xi^TU(\bar U^TU)^{-1}\bar U^T (I - UU^T)\big\lvert U, v\right] = 0
	 \end{align}
	and $\vartheta_2$ holds since $\mathbb{E}\left[\xi\xi^T\right] = \mathbb{I}_{n}$. 

	The second argument following by similar argument 
	\begin{align}
		\mathbb{E}[\Delta_2\big \lvert U] &= \mathbb{E}_{v}\left[\mathbb{E}\left[\xi^TU(I - \bar U\bar U^T)v_{\parallel}\big\lvert U, v\right]\right] + \trace\left((I - \bar U\bar U^T)(I - UU^T)\mathbb{E}\left[\xi\xi^T\right]U\right) = 0  \nn 
	\end{align}
\end{proof}}

\end{document}